\documentclass[12pt]{amsart}

\usepackage{xcolor}
\usepackage[normalem]{ulem}
\usepackage{letltxmacro}
\usepackage[all]{xy}
\usepackage{comment}
\usepackage{mathalpha}
\usepackage{amsmath,amssymb,setspace,nicefrac, yhmath, amscd}
\usepackage[active]{srcltx}
\usepackage[pagebackref, colorlinks, linkcolor=blue, citecolor=blue, urlcolor=blue, hypertexnames=true]{hyperref}
\usepackage{amsrefs}
\usepackage{dirtytalk}
\usepackage{xcolor}
\allowdisplaybreaks
\usepackage[all]{xy}

\theoremstyle{plain}
\newtheorem{theorem}{Theorem}[section]
\newtheorem{corollary}[theorem]{Corollary}
\newtheorem{proposition}[theorem]{Proposition}
\newtheorem{lemma}[theorem]{Lemma}
\theoremstyle{definition}
\newtheorem{remark}[theorem]{Remark}
\newtheorem{example}[theorem]{Example}

\newtheorem{definition}[theorem]{Definition}

\theoremstyle{plain}
\newtheorem{question}{Question}

\newcommand{\bC}{{\mathbb{C}}}

  \newcommand{\A}{{\mathcal{A}}}
  \newcommand{\B}{{\mathcal{B}}}

  \newcommand{\G}{{\mathcal{G}}}

\newcommand{\GG}{\gpdG^{(0)}}

\newcommand{\gpdG}{\mathcal{G}}

\renewcommand{\phi}{\varphi}
\newcommand{\upchi}{{\raise.35ex\hbox{\ensuremath{\chi}}}}

\newcommand{\op}{\mathrm{op}}

\newcommand{\supp}{\operatorname{supp}}

\usepackage{xcolor}

\usepackage{hyperref}

\title[Fej\'er property and Galois correspondence]{Fej\'er property and Galois correspondence for groupoid C*-algebras}
\author[Anshu]{Anshu}
\address{University of Ulsan, Department of Mathematics and Data Science, 93 Daehak-ro, Nam-gu, Ulsan, 44610 South Korea}
\email{anshu@ulsan.ac.kr}
\author[T.Amrutam]{Tattwamasi Amrutam}
\address{Institute of Mathematics of the Polish Academy of Sciences, ul.~\'Sniadeckich 8, 00--656 Warszawa, Poland}
\email{tattwamasiamrutam@impan.pl}
\author[P. Karmakar]{Pradyut Karmakar}
\address{Sam Houston State University, 332 G LDB, 1900 Avenue I ,Huntsville, Texas, USA}
\email{karmakar.pradyut@gmail.com}
\date{\today}
\begin{document}
\begin{abstract}
We introduce a notion of the Fej\'er property for topological étale groupoids. As a consequence, we show that when $\gpdG$ is a principal étale second countable groupoid satisfying the Fej\'er property, every closed $C_0(\GG)$-bimodule $M\subset C_r^*(\gpdG)$ is of the form $\overline{C_c(U)}^r$ for some open set $U$. Moreover, we get a Galois correspondence in the sense that every intermediate $C^*$-algebra $\B$ with $C_0(\GG)\subseteq \B\subseteq C_r^*(\gpdG)$ is of the form $C_r^*(\mathcal{H})$ for some open subgroupoid $\mathcal{H}\leq \gpdG$. 
\end{abstract}
\maketitle
\tableofcontents
\addtocontents{toc}{\protect\setcounter{tocdepth}{0}}
\addtocontents{toc}{\protect\setcounter{tocdepth}{1}}

\section{Introduction}
The approximation property (AP) for groups was introduced by Haagerup and Kraus \cite{haagerup1994approximation}. They also showed that the associated reduced group $C^*$-algebra has this property if the group is weakly amenable (a weaker notion than that of amenability). 
Crann and Neufang~\cite{CN22} (also see~\cite{suzuki2017group}) recently proved that the AP property for groups is equivalent to a non-commutative Fej\'er theorem for associated $\mathrm{C}^*$-dynamical (or $\mathrm{W}^*$) systems. They showed that the Fej\'er type representation of elements of a crossed product $\mathrm{C}^*$-algebra forces the acting group to have AP. 
Now, reduced crossed products arising from a topological dynamical system $(X,\Gamma)$ can be viewed as the reduced $C^*$-algebra associated with the transformation groupoid $\Gamma \rtimes X$.

Let $\gpdG$ be an \'etale locally compact second countable Hausdorff groupoid. Let $U=\{\gamma \in \gpdG: f(\gamma) \neq 0\}$ be the open support of $f \in C_r^{\ast} (\gpdG)$. We identify the reduced groupoid $\mathrm{C}^*$-algebra $C_{r}^{\ast}(\gpdG)$ as a subspace of $C_0(\gpdG)$, via the evaluation map $j$ (see Section~\ref{sec:prelim}). One can view $f(\gamma)$, for $\gamma \in \gpdG$, as $\gamma$-th Fourier coefficient of $f$. In this regard, we ask whether $f$ can be recovered from its Fourier coefficients.
The third author and Fuller asked the following question in~\cite{fuller2024fourier}.
\begin{question}\cite[Question~1]{fuller2024fourier}
\label{Question:fromFullerKarm}
Given $f \in C_r^{\ast}(\gpdG)$, does there exist a compactly supported function on $U$ which approximates $f$ in the reduced norm?    
\end{question} 
In this article, we explore Question~\ref{Question:fromFullerKarm} for groupoids admitting \emph{Fej\'er property}.
The following definition is motivated by \cite[Subsection~5.2]{CN22}.

\begin{definition}\label{defn:Fejer}
    Let $\gpdG$ be a topological \'etale groupoid. We say $\gpdG$ has \emph{Fej\'er property} if there exists a net $(h_i)_{i\in I} \in C_c(\gpdG)$ such that for every $f \in C_r^{\ast}(\gpdG)$, we have
    $$\|M_{h_i}(f)-f\|_{r} \xrightarrow{i\to\infty} 0.$$
    Here, $M_{h_i}(f)=h_i f$ is a pointwise product for $f \in C_c(\gpdG)$. Note that it follows from \cite[Proposition 3.3]{renault1997fourier} that each $M_{h_i}$ is a completely bounded operator. Moreover, if the sequence of multipliers $M_{h_i}$ is uniformly bounded, i.e., $\sup_{i}\|M_{h_i}\|_{cb} < \infty$ then we say $\gpdG$ has \emph{bounded Fejér property}. 
\end{definition}

As mentioned before, for a group $\Gamma$ acting, by homeomorphisms, on a locally compact Hausdorff space $X$, the reduced crossed product $C_0(X) \rtimes_r \Gamma$ can be realized as the reduced groupoid $\mathrm{C}^*$-algebra $C^*_r(\Gamma \rtimes X)$ associated to the transformation groupoid $\Gamma \rtimes X$ (see Example~\ref{subsec:tranformationgroupoid} for more on this). 
Affirmative answers to Question~\ref{Question:fromFullerKarm} concerning crossed products are known when $\Gamma$ is weakly amenable (\cite[Theorem~5.6]{BedCon2015}), or more generally, when $\Gamma$ has AP (see~ \cite{CN22,suzuki2017group}). For a Hausdorff, locally compact groupoid $\gpdG$, positive answers to Question~\ref{Question:fromFullerKarm} are known when $\gpdG$ is amenable (see \cite[Theorem~4.2]{ brown2024intermediatesubalgebrascartanembeddings}). Furthermore, if $f \in C^*_r(\gpdG)$ has a clopen support, an affirmative answer is given by \cite[Lemma~3.8]{DGRW2020}. 

In this paper, we show that the \emph{Fej\'er property} is weaker than weak amenability; namely, we establish that weakly amenable groupoids possess the \emph{bounded Fej\'er property}. This result precisely resonates with the classical case for groups. 
\begin{theorem}
\label{thm: boundedFejerfromweakamenability}
If $\gpdG$ is a weakly amenable groupoid, then $\gpdG$ has the bounded Fej\'er property.
\end{theorem}
It should come as no surprise then that the crossed products of the form $C_0(X) \rtimes_{r} \Gamma$ satisfy the above theorem, where $\Gamma$ has AP (see Proposition \ref{APF}).  Nonetheless, there are situations where Question~\ref{Question:fromFullerKarm} has a negative answer (for instance \cite[Example~2.4]{fuller2024fourier}), and the landscape for non-amenable groupoids remains largely unexplored. Some progress was made in this direction by considering groupoids with the Rapid decay property in ~\cite{fuller2024fourier} and for weakly amenable groupoids in \cite{pacheco2025weaklyamenablegroupoids}. However, when the groupoid fails to be inner exact, Question~\ref{Question:fromFullerKarm} can have a negative answer (see  \cite{fuller2024fourier}, for example, or \cite{HLS}). To recover elements of the groupoid $C^{\ast}$-algebra from the Fourier series, the inner exact condition is mandatory. In our setup, when the groupoid $\gpdG$ satisfies the \emph{Fej\'er property}, we show that it must be inner exact (Theorem~\ref{thm:innerexact}).

In the setting of crossed product $\mathcal{A}\rtimes_r\Gamma$, $\mathcal{A}$ being a $\Gamma$-$C^*$-algebra and $\Gamma$ a discrete group, Choda~\cite{choda1979correspondence} demonstrated a bijective correspondence between the subgroups of $\Gamma$ and a particular class of intermediate $\mathrm{C}^*$-algebras $\A \subseteq \B \subseteq \A \rtimes_r \Gamma$. An intermediate algebra $\B$ belongs to this class if, among other requirements, a faithful conditional expectation exists from $\A \rtimes_r \Gamma$ onto $\B$. Building on Izumi's earlier work~\cite{izumi2002inclusions} for finite groups, Cameron and Smith~\cite{cameron2019galois} showed that when $\A$ is a simple $\mathrm{C}^*$-algebra and $\Gamma$ acts on $\A$ by outer automorphisms, the assignment $H \mapsto \A \rtimes_{r} H$ yields a bijection between subgroups $H \subseteq \Gamma$ and $\mathrm{C}^*$-algebras $\B$ with $\A \subseteq \B \subseteq \A \rtimes_{r} \Gamma$. Analogously, we have established a Galois correspondence for principal \'etale groupoids that admit the \emph{Fej\'er property} (see~\cite[Theorem~4.5]{brown2024intermediatesubalgebrascartanembeddings} for Galois correspondence in the amenable setup).
\begin{theorem}\label{gcp1}
Let $\gpdG$ be a principal \'etale groupoid satisfying Fejér property. Let
\[
C_0(\gpdG^{(0)}) \subseteq \B \subseteq C^*_r(\gpdG) \subseteq C_0(\gpdG)
\]
be an intermediate subalgebra. Then $\B = C^*_r(\mathcal{H})$ for some open subgroupoid $\mathcal{H}$ of $\gpdG$ such that $\mathcal{H}^{(0)}=\gpdG^{(0)}$.
\end{theorem} 
In \cite[Corollary ~5.6]{brown2024intermediatesubalgebrascartanembeddings}, the authors have obtained a spectral theorem for bimodules in the crossed product, $C(X) \rtimes_{r} \Gamma$, abelian $\mathrm{C}^*$-algebra $C(X)$. Specifically, if a discrete group $\Gamma$ acts on a compact Hausdorff space $X$, then any norm closed $C(X)$ bimodule $M \subset C(X) \rtimes_{r} \Gamma$ is precisely of the form $M=\overline{\{fu_g: f \in C(X), fu_g \in M\}}$. In other words, the spectral theorem of bimodules states that one can recover the bimodule elements by examining their Fourier coefficients in the ambient bimodule.
We have generalized this result to the second countable principal \'etale groupoid admitting the Fej\'er property.  
\begin{theorem}
\label{thm:opensetclosure}
Let $\gpdG$ be a principal \'etale second countable groupoid satisfying the \emph{Fej\'er property}. Assume $M \subset C^*_r(\gpdG) \subseteq C_0(\gpdG)$ is a closed $C_0(\gpdG ^{(0)})$-bimodule. Then there exists an open set $U$ in $\gpdG$ such that $M = \overline{C_c(U)}^{\mathrm{r}}$.
\end{theorem}
\addtocontents{toc}{\protect\setcounter{tocdepth}{0}}
\subsection*{Acknowledgements} The first author was supported by the National Research Foundation of Korea (NRF) grant funded by the Korea government (MSIT) RS-2025-00557662. We want to thank Adam Fuller for discussions on the proof of Theorem \ref{thm:opensetclosure}. We thank Adam Skalski, Yongle Jiang, Aaron Kettner, and Kang Li for taking the time to read through a draft of this paper and for their helpful comments and suggestions. We thank Pawel Sarkowicz for providing us with comprehensive feedback on an earlier draft version. The third author would like to thank Narutaka Ozawa for helpful conversations, and Iason Moutzouris for carefully reading the article and providing valuable feedback. 
\addtocontents{toc}{\protect\setcounter{tocdepth}{1}}

\section{Preliminaries} 
\label{sec:prelim}
Throughout this paper, $\gpdG$ will denote a second countable locally compact Hausdorff \'etale groupoid. Since we work with the reduced group $\mathrm{C}^*$-algebra associated with a second countable locally compact Hausdorff \'etale groupoid, we now proceed to give a primer on these and refer the readers to \cite{sims2017etale} for more details. We shall follow the following notation. We denote by $\GG$ the set $\GG=\{t^{-1}t:t\in \gpdG\}$. The range and source maps from $\gpdG$ to $\GG$, denoted respectively by $r$ and $s$, are defined by 
\[
r(\gamma)=\gamma \gamma^{-1}\text{ and } s(\gamma)=\gamma^{-1}\gamma,~\gamma\in \gpdG.
\]
Note that $\gpdG_x=\{t\in \gpdG: s(t)=x\}$ is discrete for every $x\in \GG$.
A topological groupoid is a groupoid endowed with a topology such that the multiplication and inversion operations are continuous. Moreover, $\gpdG$ is said to be an \'etale groupoid if the range map $r \colon \gpdG \rightarrow \GG$ is a local homeomorphism. Note that $s$ is also a local homeomorphism. Given an \'etale groupoid $\gpdG$, the collection of all compactly supported functions on $\gpdG$, denoted by $C_c(\gpdG)$ has a $\ast$-algebra structure given by the convolution product, namely
\begin{align*}
    & f \ast g (\gamma)=\sum_{\alpha \beta =\gamma}f(\alpha) g(\beta) \qquad f, g \in C_c(\gpdG),~\alpha,\beta, \gamma\in \gpdG,\\
    & f^{\ast}(\gamma)=\overline{f(\gamma^{-1})}.
\end{align*}
For each $x\in \GG$, there is a $*$-representation $\pi_x: C_c(\gpdG)\to\mathbb{B}(\ell^2(\gpdG_x))$ defined by
\[
\pi_x(f)\delta_{\gamma}=\sum_{t\in G_{r(\gamma)}}f(t)\delta_{t\gamma},~f\in C_c(\gpdG),~\gamma\in \gpdG_x.
\]
The reduced $\mathrm{C}^{\ast}$-algebra of $\gpdG$, denoted by $C_r^{\ast}(\gpdG)$, is defined as the completion of 
\[
\left(\bigoplus_{x\in \GG}\pi_x\right)C_c(\gpdG)\subset \bigoplus_{x\in \GG}\mathbb{B}(\ell^2(\gpdG_x)).
\] 
We will denote the $\mathrm{C}^*$-norm on $C_r^*(\gpdG)$ by $\|\cdot\|_{r}$ and the space of continuous functions vanishing at infinity by $C_0(\gpdG)$.

More often than not, we shall use the following identification, which is well-known. Let $\gpdG$ be a topological \'etale groupoid and $f\in C^*_r(\gpdG)$. We define the evaluation map $j: C_r^*(\gpdG)\to C_0(\gpdG)$ by 

$$j(f)(g) = \left\langle\delta_g, \pi_{s(g)}(f)\delta_{s(g)}\right\rangle,  ~g\in \gpdG,$$
where $\delta_{g}(g)=1$ and $\delta_{g}(h)=0$ for $h \neq g$, $h \in \gpdG$. Then $j: C_r^*(\gpdG)\to C_0(\gpdG)$ is a norm-decreasing injective linear map. Moreover, $j$ is an identity map on $C_c(\gpdG)$. Since $j: C_r^*(\gpdG)\to C_0(\gpdG)$ is injective, we often identify $j(f)$ with $f$, hence, regard $f$ as
a function on $\gpdG$.
To be pedantic, for $f \in C_r^{\ast}(\gpdG)$ we denote its open support by
$$\supp(f)=\{\gamma \in \gpdG: j(f)(\gamma) \neq 0\}.$$
We suppress $j$ after the identification onwards.
Let $\gpdG$ be a topological groupoid with unit space \( \GG \). A locally compact Hausdorff \'etale groupoid $\gpdG$ is called \emph{topologically principal} if the set
\[
    \{ x \in \GG : \operatorname{Iso}(x) = \{x\} \}
\]
is dense in \( \GG \), where \( \operatorname{Iso}(x) := \{ \gamma \in \gpdG: s(\gamma) = r(\gamma) = x \} \) is the isotropy group at \( x \). 
We end this section with the following well-known example of transformation groupoids. 
\begin{example}[Transformation Groupoids]
\label{subsec:tranformationgroupoid} 
Let $X$ be a locally compact Hausdorff space and assume that a discrete group $\Gamma$ is acting on $X$, i.e., $\Gamma\curvearrowright X$ by homeomorphisms.
Then for the transformation groupoid $\Gamma \rtimes X$, we have $s(\gamma, x)=x$ and $r(\gamma, x)=\gamma  x$.
The multiplication and inverse are given by 
$(\gamma_1, \gamma_2 x) (\gamma_2, x)=(\gamma_1 \gamma_2, x)$ and $(\gamma,x)^{-1}=(\gamma^{-1}, \gamma x)$ for all $\gamma_1, \gamma_2 \in \Gamma$ and $x \in X$.
Then it is well-known that the groupoid $\mathrm{C}^{\ast}$-algebra $C_r^{\ast}(\Gamma \rtimes X) \cong C_0(X) \rtimes_r \Gamma$.
\end{example}
\section{Groupoids and Fej\'er property}
We now formalize the Fejér property for étale groupoids, modeled on the Fejér-type approximation for groups established by Crann and Neufang. Note that a function $h \colon \mathcal{G} \to \bC$ is called a completely bounded multiplier if the multiplier map $$
M_{h}(f)(\gamma)=h(\gamma) \cdot f(\gamma) \qquad \gamma \in \gpdG,~ f \in C_c(\gpdG)
$$
extends to a completely bounded linear map on $C_r^{\ast}(\gpdG)$.  
\begin{definition}\label{defn:sect3Fejer}
    Let $\gpdG$ be a topological \'etale groupoid. We say $\gpdG$ has \emph{Fej\'er property} if there exists a net $(h_i)_{i\in I} \in C_c(\gpdG)$ such that for every $f \in C_r^{\ast}(\gpdG)$, we have
    $$\|M_{h_i}(f)-f\|_{r} \xrightarrow{i\to\infty} 0,$$
    where $M_{h_i} \colon C_r^{\ast}(\gpdG) \to C_r^{\ast}(\gpdG)$ is the multiplier map.
    We identify groupoid $\mathrm{C}^{\ast}$-algebra $C_r^{\ast}(\gpdG) \subset C_0(\gpdG)$. Moreover, if the sequence of multipliers $M_{h_i}$ is uniformly bounded, i.e., $\sup_{i}\|M_{h_i}\|_{cb} < \infty$ then we say $\gpdG$ has \emph{bounded Fejér property}.
\end{definition}
In general, for $h \in C_c(\gpdG)$, using \cite[Proposition 3.3]{renault1997fourier}, we see that $$\|M_{h}\|_{\op}=\sup_{f\in C_r^*(\G)_1}\|M_h(f)\|_r < \infty $$ and is a completely bounded linear operator on $C_r^{\ast}(\gpdG)$.
The following proposition shows that the above notion is aligned with the classical case of groups admitting the approximation property (see \cite[Subsection~5.2]{CN22}).
The proof is essentially the same as in \cite[Theorem~4.10]{CN22}. We spell it out nonetheless.
\begin{proposition}\label{APF}
Let $\Gamma$ be a discrete group with the Approximation Property (AP), and let $X$ be a compact Hausdorff $\Gamma$-space. Then the transformation groupoid $\Gamma \rtimes X$ has the bounded Fej\'er property.
\end{proposition}

\begin{proof}
Since $\Gamma$ has AP, using \cite[Theorem 4.10]{CN22}, we can find a net $h_i \in C_c(\Gamma) \cap A(\Gamma)$ converging to $1$ in the weak* topology of $M_{\mathrm{cb}}(A(\Gamma))$ such that 
\begin{align}\label{limiteq}
    a=\lim_{j} \sum_{\gamma\in\Gamma} h_j(\gamma)E(au_{\gamma}^{\ast})u_{\gamma},~a\in C(X)\rtimes_r\Gamma.
\end{align}
Here, $u$ is the left regular representation of the discrete group $\Gamma$ on the crossed product.
By  \cite[Theorem 5.4]{CN22}, the associated net $H_j \colon \Gamma \times C(X) \to C(X)$ is given by 
\begin{align*}
    H_j(\gamma, f)=h_j(\gamma)f, ~\gamma\in\Gamma,~f \in C(X)\,
\end{align*}
such that
\begin{align}\label{multiplier defn}
    M_{H_j}\left(\sum_{\gamma} f_{\gamma} u_{\gamma}\right)=\sum_{\gamma} H_j(\gamma, f_{\gamma}) u_{\gamma}= \sum_{\gamma} h_j(\gamma)f_{\gamma} u_{\gamma}\qquad \gamma \in \Gamma, ~f_{\gamma} \in  C(X)\,.
\end{align}
It follows that from Equation~\eqref{limiteq} that, one has
\begin{align}\label{imptcross}
    a=\lim_{j} \sum_{\gamma\in\Gamma} h_j(\gamma)E(au_{\gamma}^*)u_{\gamma}=\lim_{j}M_{H_j}\left(\sum_{\gamma\in\Gamma}E(au_{\gamma}^*)u_{\gamma}\right).
\end{align}
For given $a \sim \sum_{\gamma} f_{\gamma} u_{\gamma} \in C(X) \rtimes \Gamma$, we define $\hat{f} \in C_c(\Gamma, C(X))$ by
$\hat{f}(\gamma)=f_{\gamma}$. We use the identification $C_c(\Gamma\rtimes X)\cong C_c(\Gamma, C(X))$ and for $h_j \in C_c(\Gamma) \cap A(\Gamma)$, we define $k_j \in C_c(\Gamma, C(X))$ as
\begin{align*}
    k_j(\gamma) = h_j(\gamma) 1_{C(X)} \qquad \gamma \in \Gamma\,.
\end{align*} 
The corresponding multipliers are defined as
\begin{align*}
M_{k_j} \colon C_c(\Gamma, C(X)) \to C_c(\Gamma, C(X)),
\end{align*}
\begin{align*}
M_{k_j}(\hat{f})(\gamma) = k_j(\gamma) \hat{f}(\gamma) = h_j(\gamma) f_\gamma = M_{H_j}(\hat{f})(\gamma)\,.
\end{align*}
under the identification $C_r^*(\Gamma\rtimes X)\cong C(X)\rtimes_r\Gamma$.
Since $M_{H_j}(a) \to a$ in norm for all $ a \in C(X) \rtimes_{r} \Gamma$, we have $M_{k_j}(f) \to f$ in $C_r^{\ast}(\Gamma \rtimes X)$ when $j \to \infty$.
Moreover $\sup_j \|M_{h_j}\|_{\mathrm{cb}} \leq 1$, gives $\sup_j\|M_{k_j}\|_{\mathrm{cb}} \leq 1$.
\end{proof}
The following example illustrates how the Fejér property behaves under disjoint unions of groupoids. We will use it later in the context of inner exactness.
\begin{example}\label{greenred} Let $\Gamma_1$, and $\Gamma_2$ be two discrete countable groups. 
Consider the groupoid $\gpdG = \Gamma_1 \sqcup \Gamma_2$, and assume that $\G$ has the Fej\'er property, then each $\Gamma_j$ also has the Fej\'er property. Indeed, Since $\gpdG$ is the disjoint union of $\Gamma_1$ and $\Gamma_2$, the reduced groupoid $\mathrm{C}^*$-algebra decomposes as
\begin{align*}
C_r^*(\gpdG) \cong C_r^*(\Gamma_1) \oplus C_r^*(\Gamma_2).
\end{align*}
Moreover, any $f \in C_c(\gpdG)$ decomposes as $f = (f_1, f_2)$ with $f_j \in C_c(\Gamma_j)$. Similarly, the net $(h_i) \subset C_c(\gpdG)$ decomposes as
\[
h_i = \bigl(h_i^{(1)}, h_i^{(2)}\bigr), \quad h_i^{(j)} \in C_c(\Gamma_j).
\]
The multiplication operators act as
$M_{h_i}(f) = \bigl(M_{h_i^{(1)}}(f_1), \, M_{h_i^{(2)}}(f_2)\bigr)$, where the multiplier $M_{h_i^{(j)}}(f_j)(\gamma) = h_i^{(j)}(\gamma) f_j(\gamma)$ for $\gamma \in \Gamma_j$. Now, the norm converges as
\[
\| M_{h_i}(f) - f \|_{r,\Gamma_1} = \max \bigl( \| M_{h_i^{(1)}}(f_1) - f_1 \|_{r, \Gamma_1}, \, \| M_{h_i^{(2)}}(f_2) - f_2 \|_{r, \Gamma_2} \bigr) \to 0.
\]
Thus, for each $j=1,2$,
$\| M_{h_i^{(j)}}(f_j) - f_j \|_{r,\gpdG} \to 0, \quad \text{for all } f_j \in C_{r}^*(\Gamma_j)$.
Moreover, since $M_{h_i}$ is completely bounded on $C_r^*(\G)$, and the completely bounded norm satisfies
\[
\| M_{h_i} \|_{cb} = \max \left( \| M_{h_i^{(1)}} \|_{cb}, \, \| M_{h_i^{(2)}} \|_{cb} \right),
\]
we see that $(h_i^{(j)})$ is a Fej\'er net for $\Gamma_j$.
\end{example}
It is natural to ask if $\Gamma \rtimes X$ having \emph{Fej\'er} has any relation with the approximation property of the group $\Gamma$. We show that this is the case if $X$, in addition, admits an invariant measure $\mu$.

\begin{theorem}
Let $X$ be a $\Gamma$-space. Assume that $\Gamma \rtimes X$ has Fej\'er property in our sense, and suppose that $\mu$ is $\Gamma$-invariant measure on $X$. Then $\Gamma$ has property AP.
\end{theorem}
\begin{proof}
Assume that $\Gamma \rtimes X$
 has Fej\'er realized by the net $\{h_j\}$. Let $\mu$ be a $\Gamma$-invariant measure on $X$.
For each $h_j\in C_c(\Gamma \rtimes X)$,  define $\tilde{h}_j \in C_c(\Gamma)$ by $$\tilde{h}_j(s)=\int_X h_j(s,x) d\mu(x),~ s \in \Gamma.$$ 
Since $\mu$ is a $\Gamma$-invariant measure, the map $E_{\mu}: C(X)\rtimes_r\Gamma\to C_r^*(\Gamma)$ defined by $f\lambda(s)\to \mu(f)\lambda(s)$ is a $\Gamma$-equivariant conditional expectation (see~\cite[Exercise~4.1.4]{brown2008textrm}).
Letting $M_{\tilde{h}_j}=E_{\mu} \circ M_{h_j}$, we see that 
 \begin{align*}
\|M_{\tilde{h}_j}(a)-a\|_{r}
&=\|M_{\tilde{h}_j}(a)-E_{\mu}(a)\|_{r}\\
&=\|E_{\mu} \circ M_{h_j}(a)-E_{\mu}(a)\|_{r} \leq \|M_{h_j}(a)-a\|_{r} \to 0.
    \end{align*}
Also, note that $\|M_{\tilde{h}_j}\|_{cb}\le \|M_{h_j}\|_{cb}$ for each $j$.  Hence, $\sup_{j}\|M_{\tilde{h}_j}\|_{cb} < \infty$. Thus, $\Gamma$ has AP.
\end{proof}
\begin{remark}
We do not know if the assumption that $\mu$ is invariant can be removed in general.
\end{remark}
We now show that when $\G$ has the \emph{Fej\'er property}, the support of an element in $C_r^*(\gpdG)$ is contained in the subalgebra generated by an open set. In particular, the support of an element determines the subalgebra it belongs to. The following proposition makes it precise. The idea of the proof is essentially contained in \cite[Theorem]{brown2024corrigendum}. We sketch it nonetheless with all its details. Recall that a subset $B$ of an \'etale groupoid is called a bisection if there is an open set $U$ containing $B$ such that $r: U \to r(U)$ and $s:U \to s(U)$ are homeomorphisms.
\begin{proposition}\label{prop:main containment}
Assume that $\gpdG$ has the Fejér property. Let $U$ be an open subset of $\gpdG$ and let $f \in C_r^{\ast}(\gpdG)$ be such that $\supp(f) \subseteq U$. Then, $f \in \overline{C_c(U)}^{r}$
\end{proposition}
\begin{proof}
Since $\gpdG$ has the \emph{Fejér property}, there exists a net $(h_i)\in C_c(\gpdG)$ satisfying the conditions in Definition~\ref{defn:Fejer}. Hence,  it suffices to show that $M_{h_i}(f) \in \overline{C_c(U)}^{r}$.
  Observe that $M_{h_i}(f) \in C_c(\gpdG)$ for $f \in C_r^{\ast}(\gpdG)$. Since $M_{h_i}(f) \in C_c(\gpdG)$ vanishes off $U$, we can assume that $f \in C_c(\gpdG)$.
   Therefore, assume that $f \in C_c(\gpdG)$ and we show that $f\in \overline{C_c(U)}^r$.
    Let $E$ be a compact subset of $\gpdG$ outside of which $f$ vanishes. 
    
    Since $\gpdG$ is étale we may find finitely many open bisections $\{U_1, U_2, \ldots U_n\}$, whose union contains $E$. Let $\{g_1 \ldots,g_n\}$ be a partition of unity for $E$ subordinate to the open cover $\{U_1,\ldots,U_n\}$. That is, each $g_i$ is a continuous nonnegative real-valued function on $G$, $\supp(g_i)\subset U_i$, and such that 
    $$f=\sum_{i=1}^n g_i f,$$
where $g_if$ is the pointwise product. Towards this end, it is enough to show that $g_if \in \overline{C_c(U)}^{r}$ for each $i$. Note that $g_if$ vanishes off the bisection $U_i \cap U$. Thus, we may assume, without loss of generality, that our originally chosen $f$ vanishes outside
some open bisection $U_0$, with $U_0 \subseteq U$.
Set $W = s(U_0)$, and let $(w_j)_j$ be an approximate unit for $C_0(W)$ contained in $C_c(W)$.
Observe that, for each $j$, we have 
\begin{align*}
  f  \ast w_j(g) = f(g)w_j(s(g)) \qquad g \in \gpdG\,. 
\end{align*}
Since $s$ is a homeomorphism between $U_0$ and $W$, it follows that
$$f \ast w_j = f ( w_j \circ s ) \to f$$
uniformly. Since $U_0$ is a bisection it follows that $f \ast w_j\to f$ in $\|\cdot\|_r$ (see \cite[Theorem~5.1.11]{sims2017etale}). Observe that $f \ast w_j$ is compactly supported in $U_0 \subseteq U$. Then it follows that $f \ast w_j \in C_c(U_0) \subseteq C_c(U) \subseteq \overline{C_c(U)}^{r}$, and hence $f \in \overline{C_c(U)}^{r}$.
\end{proof}
\subsection{Inner exactness} To analyze structural consequences of the Fejér property, we recall the notion of inner exactness for reduced groupoid $\mathrm{C}^*$-algebras.
Let $\gpdG$ be a topological groupoid with unit space $\GG$. A subset $U \subseteq \GG$ is called \emph{invariant} if for all $\gamma \in \gpdG$, 
\[
s(\gamma) \in U \implies r(\gamma) \in U,
\]
where $s$ and $r$ denote the source and range maps, respectively. Equivalently, $U$ is invariant if $s^{-1}(U) = r^{-1}(U)$, i.e., every arrow in $\gpdG$ with source in $U$ also has range in $U$ (and vice versa).
Given an invariant subset $U \subseteq \GG$, the \emph{reduction} or \emph{restricted subgroupoid} $\gpdG |_U$ is defined as
\[
\gpdG |_U := \{ \gamma \in \gpdG : s(\gamma) \in U \text{ and } r(\gamma) \in U \}.
\]
This is a subgroupoid of $\gpdG$ with unit space $U$.
Let $F \subseteq \GG$ be a closed invariant subset and $U:= \GG \setminus F$. 
The inclusion $C_c(\gpdG|_U) \hookrightarrow C_c(\gpdG)$ induces injective $*$-homomorphisms
\[
C^*(\gpdG|_F) \to C^*(\gpdG) \quad \text{and} \quad C^*_r(\gpdG|_F) \to C^*_r(\gpdG).
\]
Moreover, the restriction map $C_c(\gpdG) \to C_c(\gpdG|F)$ induces surjective $*$-homomorphisms
\[
C^*(\gpdG) \to C^*(\gpdG|_F), \qquad C^*_r(\gpdG) \to C^*_r(\gpdG|_F).
\]
Consequently, the following sequence is exact:
\[
0 \longrightarrow C^*(\gpdG|_U) \longrightarrow C^*(\gpdG) \longrightarrow C^*(\gpdG|_F) \longrightarrow 0.
\]
However, this exactness may fail for the reduced $\mathrm{C}^*$-algebra; this can be demonstrated by considering the so-called HLS groupoid (see, for example~\cite{HLS}).

As mentioned above, there exists a sequence
\begin{equation} \label{eq:inner-exact-sequence}
0 \longrightarrow C^*_r(\gpdG|_U) \xrightarrow{\iota} C^*_r(\gpdG) \xrightarrow{\pi} C^*_r(\gpdG|_F) \longrightarrow 0,
\end{equation}
which is not exact in the middle. 

\begin{definition}
    A groupoid $\gpdG$ is said to be \emph{inner exact} if the sequence in Equation~\eqref{eq:inner-exact-sequence} is exact for all closed invariant subsets $F \subseteq X$.
\end{definition}
\begin{example}\label{lem:innexctgpid}
Let $\Gamma_1$ and $\Gamma_2$ be discrete groups, and consider the groupoid 
\begin{align*}
\gpdG = \Gamma_1 \sqcup \Gamma_2,
\end{align*}
the disjoint union of $\Gamma_1$ and $\Gamma_2$ viewed as groupoids with one object each. Observe that 
$\GG = \{e_{\Gamma_1}\} \sqcup \{e_{\Gamma_2}\}$, where $e_{\Gamma_i}$ corresponds to the unit elements of the groups $\Gamma_i$. The only open $\gpdG$-invariant subsets of $\GG$ are $\emptyset$, $\{e_{\Gamma_1}\}$, $\{e_{\Gamma_2}\}$, and $\GG$.
For each such open invariant subset, the restricted groupoid $\gpdG|_U$ is either empty or isomorphic to one of the groups $\Gamma_1$ or $\Gamma_2$ viewed as groupoids. Moreover, the reduced groupoid $\mathrm{C}^*$-algebra decomposes as a direct sum $ C_r^*(\Gamma_1) \oplus C_r^*(\Gamma_2)$. Considering the short sequences corresponding to each invariant open set, one checks that they are all exact. Hence $\gpdG$ is inner exact. 
\end{example}
We now show that groupoids possessing the \emph{Fej\'er property} must necessarily be inner exact. 

\begin{theorem}
\label{thm:innerexact}
    If $\gpdG$ has \emph{Fej\'er property}, then $\gpdG$ is inner exact.
\end{theorem}
\begin{proof}
Let $F \subseteq \GG$ be a closed invariant subset and $U := \GG \setminus F$. We need to show that the sequence
\[
0 \longrightarrow C^*_r(\gpdG|_U) \xrightarrow{\iota} C^*_r(\gpdG) \xrightarrow{\textcolor{teal}{\pi}} C^*_r(\gpdG|_F) \longrightarrow 0
\]
is exact. Let $f \in C^*_r(\gpdG)$ with $\pi(f) = 0$. Then $f|_{\gpdG|_F} = 0$, and since $U$ is invariant, we see that $\operatorname{supp}(f) \subseteq \gpdG|_U$. By 
Proposition~\ref{prop:main containment}, it follows that $f \in \overline{C_c(\gpdG|_U)}^r = \operatorname{im}(\iota)$. Therefore, $\ker(p) \subseteq \operatorname{im}(\iota)$. The other inclusion is clear, so the sequence is exact and $\gpdG$ is inner exact.
\end{proof}
The inner exactness does not imply the Fej\'er property for groupoids as the following example demonstrates. 
\begin{example}\label{exmpl}
Take $\Gamma_1 = SL_3(\mathbb{Z})$ and $\Gamma_2 = \mathbb{Z}$. Then 
\begin{align*}
\gpdG = \Gamma_1 \sqcup \Gamma_2
\end{align*}
is inner exact but does not have the \emph{Fej\'er property} by Example \ref{greenred} and Example~\ref{lem:innexctgpid}, since having the \emph{Fej\'er property} for $\gpdG$ would imply that  $SL_3(\mathbb{Z})$ also has the Fej\'er property. However, for any locally compact group, AP is equivalent to the Fej\'er property by \cite{CN22}. But this is a contradiction since $SL_3(\mathbb{Z})$ does not have the Approximation Property (AP) (see~\cite{deLaatdelaSalle+2018+49+69}). 
\end{example}
\subsection{Weak amenability} It is well-known in the group case that weak amenability implies the approximation property (see~\cite{haagerup1994approximation} and \cite{vergara2024invitation}). We now show that weak amenability for an \'etale groupid $\G$ implies the \emph{bounded Fej\'er property}. In particular, we prove Theorem~\ref{thm: boundedFejerfromweakamenability}. Before that, we briefly recall the definition of weak amenability and refer the readers to \cite{pacheco2025weaklyamenablegroupoids} for more details.

Given a quasi-invariant probability measure $\mu$ on $\gpdG$, a function $\varphi\in L^{\infty}(\gpdG,\mu)$ is said to be a multiplier of the Fourier algebra if the map $M_{\varphi}: A(\gpdG)\to A(\gpdG)$ given by $M_{\varphi}(f)=\varphi f$ is well-defined and bounded (see \cite[Subsection~1.3]{renault1997fourier} for definition of $A(\gpdG)$). Moreover, if $M_{\varphi}$ is completely bounded, we say that $\varphi$ is a completely bounded multiplier. We denote the set of such multipliers by $M_0A(\gpdG)$, endowed with the completely-bounded (CB)-norm,
\[
\|\varphi\|_{M_0A(\gpdG)}=\left\|M_{\varphi}\right\|_{\text{cb}}.\]
A topological \'etale groupoid $\gpdG$ is said to be \emph{(topologically) weakly amenable} if there exists a net $(\varphi_i) \subset C_c(\gpdG)$, a constant $C > 0$, and a quasi-invariant probability measure with full support $\mu$ on $\GG$, such that
   \begin{enumerate}
       \item $\varphi_i \rightarrow 1 \text{ uniformly on compact subsets of } \gpdG$, and
       \item $\sup_i \|\varphi_i\|_{M_0A(\gpdG)} = C < \infty$.
   \end{enumerate}

\begin{proof}[Proof of Theorem~\ref{thm: boundedFejerfromweakamenability}]
Let $\gpdG$ be a weakly amenable \'etale groupoid. Then, there exists a net $(\varphi_i) \subset C_c(\gpdG)$, a constant $C > 0$, and a quasi-invariant probability measure with full support $\mu$ on $\GG$, such that$ \varphi_i \rightarrow 1$ uniformly on compact subsets of  $\gpdG$, and
$\sup_i \|\varphi_i\|_{M_0A(\gpdG)} = C < \infty$. Since $\phi_i \in M_0A(\gpdG) \cap C_c(\gpdG)$, $M_{\phi_i} \colon C_r^{\ast}(\gpdG) \to C_r^{\ast}(\gpdG)$ defines a completey bounded operator. Moreover, since $\sup_i \|\varphi_i\|_{M_0A(\gpdG)} = C < \infty$, we have that $\sup_i \|M_{\phi_i}\|_{\mathrm{cb}} < \infty$.
Now for $f \in C_c(\gpdG)$ supported in bisection, and using the fact that $\|\cdot\|_{r}\le \|\cdot\|_{\infty}$ (see~\cite[Lemma~3.2.2]{sims2017etale}) for such functions, we have
\begin{align*}
    \|M_{\phi_i}(f)-f\|_{r}&=\|\phi_if-f\|_{r}\\&\le \|\phi_if-f\|_{\infty} \\&\leq M \|\phi_i-1\|_{\infty, K} \to 0, \qquad\supp(f) \subseteq K
\end{align*}
where $M = \sup_i \|M_{\varphi_i}\|_{cb}$. Now, let $f$ be in $C_c(\gpdG)$ be given. Suppose $\text{supp}(f) \subseteq L$ where $L$ is a compact subset of $\gpdG$. Then $L$ can be covered by finitely many open bisections $U_1,\dots, U_n$. Choose a partition of unity subordinate to these bisections, i.e., choose $g_1,\dots,g_n$ in $C_c(\gpdG)$ such that $\text{supp}(g_i) \subset U_i$ with $\sum_{i=1}^{n} g_i=1$. Take $f_i=fg_i$. Then we have $f=\sum_{i=1}^{n} f_i$ with $f_i\in C_c(\gpdG)$ with $\text{supp}(f_i) \subset U_i$.
We now see that
\begin{align*}
    \|M_{\phi_i}(f)-f\|_{r}&=\left\|\sum_{k=1}^{n}M_{\phi_{i}}(f_k)-\sum_{k=1}^{n}(f_k)\right\|_{r}\\
    &=\left\|\sum_{k=1}^{n}(M_{\phi_i}(f_k)-f_k)\right\|_{r}\\
    &\leq \sum_{k=1}^{n}\left\|M_{\phi_i}(f_k)-f_k\right\|_{r} \rightarrow 0.
\end{align*}
Now by a density argument and using the fact that $\sup_i \|M_{\phi_i}\|_{\mathrm{cb}} < \infty$, one has for any $f \in C_r^{\ast}(\gpdG)$, 
$$\|M_{\phi_i}(f)-f\|_r \to 0.$$
Hence $\gpdG$ has bounded \emph{Fej\'er property}. The proof is now complete.
\end{proof}
\subsection{Galois Correspondence}
In this section, we establish a version of the Galois correspondence in the context of groupoids. The next corollary, which is immediate from Proposition~\ref{prop:main containment}, shows that if an element is supported on an open subgroupoid, it lies in the corresponding reduced $\mathrm{C}^*$-subalgebra.
\begin{corollary}
    Let $\mathcal{H}$ be an open subgroupoid of $\gpdG$, where $\gpdG$ has the Fej\'er property. Assume that $f \in C_r^{\ast}(\gpdG)$ is such that $\supp(f) \subseteq \mathcal{H}$. Then $f \in C_r^{\ast}(\mathcal{H})$.
\end{corollary}
Galois correspondences are significant in the analysis of intermediate Cartan embeddings. A primary question concerns whether, for $D \subseteq B \subseteq A$ with $D$ Cartan in $A$, the subalgebra $D$ remains Cartan in an intermediate subalgebra $B$ of $A$. When $D \subseteq A$ is a Cartan subalgebra of $A$, a groupoid model is associated with this structure (see Kumjian~\cite{kumjian1986c} and Renault-Li~\cite{li2019cartan}). Specifically, there exists a twist $q \colon \Sigma \to G$ such that $D \cong C_0(\gpdG^{(0)})$ and $A \cong C_r^{\ast}(\gpdG, \Sigma)$.
We now establish a Galois correspondence below.
Nonetheless, it is essential to note that such a Galois correspondence will fail if the action is not free (for instance, see \cite[Example~5.1]{brown2024intermediatesubalgebrascartanembeddings}), that is, if the groupoid is not principal. Given any open set $U$, we write
$$A_U:=\{f \in C_r^{\ast}(\gpdG) \subseteq C_0(\gpdG): \supp(f) \subset U\}.$$
We now prove a Galois correspondence between open subgroupoids and intermediate $\mathrm{C}^*$-subalgebras for principal groupoids with the Fejér property.
\begin{corollary}
Let $\gpdG$ be a principal \'etale groupoid satisfying \emph{Fejér property}. Let
\[
C_0(\GG) \subseteq \B \subseteq C^*_r(\gpdG) \subseteq C_0(\gpdG)
\]
be an intermediate subalgebra. Then $\B = C^*_r(\mathcal{H})$ for some open subgroupoid $\mathcal{H}$ of $\gpdG$.
\end{corollary}
\begin{proof}
Define the set
\[
\mathcal{H} = \left\{ \gamma \in \gpdG : \exists\, b \in \B \text{ such that } b(\gamma) \neq 0 \right\}.
\]
Since $\G$ is a principal \'etale groupoid, it follows from \cite[Theorem~2.3]{brown2024intermediatesubalgebrascartanembeddings} that $C_0(\G^{(0)})$ is a $C^*$-diagonal of $C_r^*(\G)$. Hence, arguing similarly as in the first part of the proof of \cite[Theorem~4.4]{brown2024intermediatesubalgebrascartanembeddings}, we see that $\mathcal{H}$ is an open subgroupoid of $\gpdG$. Then we have
\[
C_r^*(\mathcal{H}) \subseteq \B \subseteq A_{\mathcal{H}}.
\]
Now, using Proposition~\ref{prop:main containment}, we have $A_{\mathcal{H}} = C^*_r(\mathcal{H})$. Consequently, we get $\B = C^*_r(\mathcal{H})$, which concludes the proof.
\end{proof}
Recall that the map  \( f \in C_c(\gpdG) \) defined by $E(f)(x) := f(x)$, for all  $x \in \GG$ extends to a conditional expectation
$E : C_r^*(\gpdG) \to C_0(\G^0)$. Moreover, given an inclusion of unital $C^*$-algebras $D \subset A$, we denote by $N(A, D)$ the set of all elements in $A$ that normalize $D$, i.e.,
\[N(A, D)=\left\{a\in A: aDa^*\cup a^*Da\subset D\right\}.\]
The following lemma constructs normalizers supported on a given bisection, a key step in describing bimodules.
\begin{lemma} \label{lem: normalizerequalssupport} Let $\G$ be a principal \'etale groupoid, and $M$ be a norm closed $C_0(\GG)$-bimodule. Moreover, let $a \in M$ be such that $\operatorname{supp}(a)$ contains a bisection $B$. Then there exists $m \in M \cap N(C^*_r(\gpdG), C_0(\GG))$ such that $\operatorname{supp}(m) = B$.
\end{lemma}

\begin{proof}
Consider $n \in C_0 (B)$ such that $\operatorname{supp}(n) = B$. Then $n$ normalizes $C_0(\GG)$ by \cite[Proposition~4.8(ii)]{renault2008cartan}. Moreover, since $B$ is a bisection, the reduced norm $\|\cdot\|_r$ and the $\sup$-norm $\|\cdot\|_{\infty}$  agree on it (see~\cite[Corollary~3.3.4]{sims2017etale}). Hence $n$ can be approximated by elements from $C_c(B)$ in $\|\cdot\|_r$-norm. Consequently, $n\in C_r^*(\G)$. We now claim that $m= n \ast E(n^{\ast} \ast a)$ is the required element. We first show that $m$ is supported on the bisection $B$. Indeed, since $E(n^{\ast} \ast a) \in C_0(\GG)$ and $n \in C_r^{\ast}(\gpdG)$, using \cite[Proposition~4.3]{renault2008cartan} along with \cite[Lemma~4.4]{renault2008cartan}, we have 
\begin{align*}
(n * E(n^* * a))(\gamma) = n(\gamma) E(n^* * a)(s(\gamma))=n(\gamma)\overline{n(\gamma)} a(\gamma)\,.
\end{align*}
It then follows immediately that the support of $m$ is $B$. Using~\cite[Proposition~3.10]{DonPit2008}, we see that $m\in M$. It now follows from  \cite[Proposition~4.8(i)]{renault2008cartan} that $m$ is a normalizer.
\end{proof}
Recall that for any open set $U$, we denote
$$A_U:=\{f \in C_r^{\ast}(\gpdG) \subseteq C_0(\gpdG): \supp(f) \subset U\}.$$
We now establish a spectral-type theorem for bimodules, showing that closed $C_0(\G^{(0)})$-bimodules correspond exactly to open subsets of the groupoid.
\begin{proof}[Proof of Theorem~\ref{thm:opensetclosure}]
We first show the following. Let $m\in M$ be a normalizer of $C_0(\GG)$. Then $C_c\left(\operatorname{supp}(m)\right) \subset M$.

Since $m$ normalizes $C_0(\GG)$, $\supp(m)$ is a bisection by \cite[Proposition~4.8(ii)]{renault2008cartan}. Consequently, the range map $r: \operatorname{supp}(m) \to \GG$ is a homeomorphism onto its image. Take $f \in C_c(\supp(m))$. Following the approach of \cite[Theorem~2.1.2]{KomPub2025}, define $g \in C_c(\GG)$  by
\[
g(x) =
\begin{cases}
\displaystyle \frac{f\left((r|_{\operatorname{supp}(m)})^{-1}(x)\right)}{m\left((r|_{\operatorname{supp}(m)})^{-1}(x)\right)} & \text{if } x \in r(\operatorname{supp}(m)), \\[2ex]
0 & \text{otherwise}.
\end{cases}
\]
Note that $f = g \ast m$, and since $M$ is a $C_0(\GG)$-bimodule, it follows that $f \in M$. This shows that $C_c(\text{supp}(m))\subset M$.

For each $a\in M$, define 
$U_a = \{\gamma \in \gpdG : a(\gamma) \neq 0\}$. Let $$ U := \bigcup_{a \in M} U_a$$
Let $B \subseteq U$ be a bisection in $\gpdG$. We now show that $C_c(B) \subset M$.
Let $f \in C_c( B)$ and let $K$ be a compact subset of $\gpdG$ such that $\supp(f) \subset K \subset B$. Observe that for any $\gamma \in K\subset B\subseteq U$, there exists $a_{\gamma} \in M$ such that $a_{\gamma}(\gamma) \neq 0$. 

Since $\gpdG$ is \'etale and has a basis of open bisections, there exists a bisection $W_\gamma$ such that $\gamma \in W_\gamma \subseteq \operatorname{supp}(a_\gamma)$. Therefore,
\[
K \subseteq \bigcup_{\gamma \in K} W_\gamma.
\]
As $K$ is compact, there exist finitely many $\gamma_1, \ldots, \gamma_n \in K$ such that
\[
K \subseteq \bigcup_{i=1}^n W_{\gamma_i}.
\]
By Lemma~\ref{lem: normalizerequalssupport}, for each $i = 1, \ldots, n$, there exists $n_{\gamma_i} \in M$ (a normalizer) such that $\operatorname{supp}(n_{\gamma_i}) = W_{\gamma_i}$. Let $\{\tilde{f}_i\}_{i=1}^n$ be a partition of unity for $K$ subordinate to the open cover $\{W_{\gamma_i} : i = 1, \ldots, n\}$. Then, letting $f_i=\tilde{f}_if$ for each $i=1,2,\ldots,n$, we have
\[
f = \sum_{i=1}^n f_i.
\]
As each $f_i \in C_c(W_{\gamma_i}) = C_c(\operatorname{supp}(n_{\gamma_i})) \subset M$, by  our first claim, it follows that $f \in M$. This shows that $C_c(B)\subset M$.

We now show that $C_c(U) \subset M$. Let $f \in C_c(U)$. Then $\operatorname{supp}(f) \subset K \subset U$, where $K$ is a compact set. For each $\gamma \in K$, there exists an open bisection $U_i$ for $i = 1, \ldots, n$ such that $K \subset \bigcup_{i=1}^n U_i$ and 
\[
f = \sum_{i=1}^n f_i
\]
with $f_i \in C_c(U_i)$. As $U_i \subset U$ for all $i = 1, \ldots, n$, the claim follows. Therefore, we see that
\[C_c(U)\subset M\subset A_U.\]
Since $\gpdG$ has the \emph{Fejér property}, it follows from Proposition~\ref{prop:main containment} that $A_U = \overline{C_c(U)}^{\mathrm{r}}$. 
Consequently, we conclude that $M = \overline{C_c(U)}^{\mathrm{r}}$.
\end{proof}
\bibliography{name}
\bibliographystyle{amsalpha}
\end{document}